\documentclass[a4paper]{amsart}
\usepackage[american]{babel}
\usepackage[T1]{fontenc}
\usepackage[utf8]{inputenc}
\usepackage{amsmath,amsfonts,amstext} 
\usepackage{amsthm,amssymb,amsmath,amscd,graphicx,epsfig,hhline,mathrsfs,latexsym,url,color}

\newcommand{\mc}{\mathcal}
\newcommand{\ms}{\mathscr}

\newcommand{\mb}{\mathbb}
\newcommand{\mr}{\mathrm}
\newcommand{\aveN}{\frac{1}{N}\sum_{n=1}^N}

\newtheorem{Rem}{Remark}

\newtheorem{Exam}{Example}
\newtheorem{Thm}{Theorem}

\theoremstyle{definition}

\theoremstyle{plain}

\DeclareMathOperator{\Fix}{Fix}
\newcommand{\lin}{\text{lin}}

\newcommand{\la}{\langle}
\newcommand{\ra}{\rangle}

\title{On the pointwise entangled ergodic theorem}
\author{Tanja Eisner}
\address{Mathematisches Institut, Universit\"at Leipzig, P.O. Box 100920, 04009 Leipzig, Germany}
\email{eisner@math.uni-leipzig.de}
\author{D\'avid Kunszenti-Kov\'acs}
\address{MTA Alfréd Rényi Institute of Mathematics, P.O. Box 127, H-1364 Budapest, Hungary}
\email{daku@fa.uni-tuebingen.de}

\keywords{Entangled ergodic averages, pointwise convergence, unimodular eigenvalues, Dunford-Schwartz operators}
\subjclass[2000]{Primary: 47A35; Secondary: 37A30}

\begin{document}

\maketitle

\begin{center}
\emph{Dedicated to  our advisor Rainer Nagel on the occasion of his 75$^{th}$ birthday }
\end{center}

\vspace{0.5cm}

\begin{abstract}
We present some twisted compactness conditions for almost everywhere convergence of one-parameter entangled ergodic averages of Dunford-Schwartz operators  $T_0,\ldots, T_a$ on a Borel probability space of the form $$
\aveN T_a^n A_{a-1}T_{a-1}^nA_{a-1}\cdot \ldots \cdot A_0 T_0^nf$$
for $f\in L^p(X,\mu)$, $p\geq 1$. We also discuss examples and present a continuous version of the result.
\end{abstract}

%
%

\section{Introduction}

For the proof of a central limit theorem for certain models in quantum pro\-ba\-bility, Accardi, Hashimoto, Obata \cite{AHO}  introduced the study of entangled ergodic averages. These were studied further 
by Liebscher \cite{liebscher:1999}, Fidaleo \cite{fidaleo:2007,fidaleo:2009,fidaleo:2010}, and the authors \cite{EKK}.
We refer to \cite{EKK} for more information and the connection to noncommutative multiple ergodic theorems.

The setting of the entangled ergodic theorems is the following.
Let $k\leq m$ be positive integers and $\alpha:\{1,\ldots,m\}\to \{1,\ldots,k\}$ be a surjective map. Let further $E$ be a Banach space, $T_1,\ldots,T_m$ and $A_1,\ldots,A_{m-1}$ be bounded operators on $E$. As shown in \cite{EKK}, the entangled ergodic averages
$$
\frac{1}{N^k}\sum_{n_1,\ldots,n_k=1}^N T_m^{n_{\alpha(m)}} A_{m-1}T_{m-1}^{n_{\alpha(m-1)}} \cdots A_1 T_1^{n_{\alpha(1)}}
$$
converge in norm under quite weak compactness conditions on the operators $T_j$ and the pairs $(A_j,T_j)$.

In our knowledge, pointwise convergence of the entangled ergodic averages for $E:=L^p(X,\mu)$, where $(X,\mu)$ is a probability space and $p\geq 1$, and for Koopman or Dunford-Schwartz operators $T_1,\ldots,T_m$ has not yet been investigated. The aim of this paper is to  close this gap partially and to present sufficient conditions in the spirit of those in  \cite{EKK} for the case $k=1$. The general case remains open. In what follows, we shall denote by $\mb{N}$ the set of positive integers.

Our main result is the following. (Recall that a  Borel probability space is a compact metrizable space with a Borel probability measure, see e.g.~Einsiedler, Ward \cite[Def.~5.13]{EW}. For the Jacobs-deLeeuw-Glicksberg decomposition and basics on Dunford-Schwartz operators see Section \ref{sec:prelim}.)

\begin{Thm}\label{thm:main}
For $a\in\mb{N}$, let $T_0,T_1,\ldots T_a$ be
Dunford-Schwartz operators 
 on a Borel probability space $(X,\mu)$ with $\Fix |T_1|=\ldots=\Fix |T_a|=\la\mathbf{1}\ra$.
For $p\in[1,\infty)$ and $E:=L^p(X,\mu)$, let 
$E=E_{0,r}\oplus E_{0,s}$ be the Jacobs-deLeeuw-Glicksberg decomposition corresponding to $T_0$, and let further $A_j\in\mc{L}(E)$ $(0\leq j< a)$ be bounded operators. For a function $f\in E$ and an index $0\leq j< a$, write $\ms{A}_{j,f}:=\left\{A_jT_j^nf\left|\right.n\in\mb{N}\right\}$. Suppose that the following conditions hold:
\begin{itemize}
\item[(A1)]\emph{(Twisted compactness)}
For every $f\in E$, $0\leq j< a$ and $\varepsilon>0$, there exists a decomposition (depending on $f$, $j$ and $\varepsilon$) $E=\mc{U}\oplus \mc{R}$ with $\dim \mc{U}<\infty$
such that $$P_\mc{R}\ms{A}_{j,f}
\subset B_\varepsilon(0,L^\infty(X,\mu)),$$
where $P_\mc{R}$ denotes the projection onto $\mc{R}$ along $\mc{U}$.
\item[(A2)]\emph{(Joint $L^\infty$-boundedness)}
There exists a constant $C>0$ such that
\[\{A_jT^n_j|n\in\mb{N},1\leq j< a\}\subset B_C(0,\mc{L}(L^\infty(X,\mu)).
\]
\end{itemize}
Then we have the following:
\begin{enumerate}
\item for each $f\in E_{0,s}$, $\frac{1}{N}\sum_{n=1}^N |T_a^nA_{a-1}T^n_{a-1}\ldots A_1T_1^nA_0T_0^n f|\rightarrow 0$ pointwise a.e.;
\item if $p=2$, then for each $f\in E_{0,r}$, $\frac{1}{N}\sum_{n=1}^N T_a^nA_{a-1}T^n_{a-1}\ldots A_1T_1^nA_0T_0^n f$ converges pointwise a.e.~to
\begin{equation}\label{eq:lim}
\sum_{\substack{\lambda_j\in\sigma_j\, (1\leq j\leq a)\\  
\lambda_1\cdot\ldots\cdot\lambda_a=1}}
P^{(a)}_{\lambda_a}A_{a-1}P^{(a-1)}_{\lambda_{a-1}}A_{a-2}\ldots A_1P^{(1)}_{\lambda_1}f,
\end{equation}
where $\sigma_j=P_\sigma(T_j)\cap \mb{T}$ and $P^{(j)}_{\lambda_j}$ is
the projection onto the eigenspace of $T_j$ corresponding to $\lambda_j$, i.e., the mean
ergodic projection of the operator $\overline{\lambda_j}T_j$.
\end{enumerate}
\end{Thm}

Note that the above conditions are stronger than the conditions in \cite{EKK} for norm convergence. (In particular, the total mean ergodicity assumption on $T_a$ follows from the discussion in Section \ref{sec:prelim}).
Since the pointwise limit coincides with the norm limit, the above representation of the limit in Theorem \ref{thm:main} is the same as in \cite[Theorem 3]{EKK}.

Note further that a sufficient condition for (A2) is that every $A_j$ is bounded as an operator on $L^\infty(X,\mu)$.

An interesting question not studied in this paper is to find  analogues of the above result for non-commutative multiple ergodic averages. While norm convergence results can just be translated into  corresponding results for convergence of non-commutative multiple ergodic averages in the strong sense, see, e.g., \cite[Section 4]{EKK},  the situation with pointwise convergence is more delicate.  Several different analogues of pointwise convergence in the non-commutative case are provided by Egorov's theorem (see e.g. Junge, Xu \cite{JungeXu}, Lance \cite{Lance}, Yeadon \cite{Yeadon} for non-commutative Birkhoff's  theorem), but the use of the uniform topology combined with projections makes a direct connection to our setting difficult.

The paper is organized as follows. After showing the main ideas in a simpler case in Section \ref{sec:model-case} and presenting the proof of Theorem \ref{thm:main} in Section \ref{sec:general-case}, we discuss some examples and the continuous case in Section \ref{sec:ex}.


\section{Notations and tools}\label{sec:prelim}

We denote by $\mb{T}$ the unit circle in $\mb{C}$.
We further denote by $\mc{N}$ the set of all bounded sequences $\{a_n\}\subset \mb{C}$ with the property
$$
\lim_{N\to\infty}\aveN |a_n|=0.
$$
By the Koopman-von Neumann lemma, see e.g.~Petersen \cite[p. 65]{petersen:1983}, $(a_n)\in\mc{N}$ if and only if it is bounded and converges to $0$ along a sequence of density $1$.

Let $E$ be a Banach space and let $T\in \mc{L}(E)$ be \emph{weakly almost periodic},
i.e., such that for every $f\in E$ the set $\{T^n f, n\in\mb{N}\}$ is relatively weakly compact in $E$. We will use the following version of the  Jacobs-deLeeuw-Glicksberg decomposition,
see \cite[Theorem II.4.8]{eisner-book} or \cite[Section 16.3]{EFHN}:
$$E=E_r \oplus E_s,$$
where
\begin{eqnarray*}
E_r&:=&\overline\lin\{f\in E:\ Tf=\lambda f\mbox{ for some }\lambda\in\mb{T}\},\\
E_s&:=&\{f\in E:\ (\varphi( T^nf))\in\mc{N} \mbox{ for every } \varphi\in E'\}.
\end{eqnarray*}
Here, $E_r$ is called the \emph{reversible} subspace and $E_s$ the \emph{stable} subspace. Note that Jacobs, deLeeuw, Glicksberg and some other authors use(d) the terminology ``flight vectors'' for elements of $E_s$. Our preference of the name ``(almost weakly) stable vectors'' is justified by the fact that the orbit of such a vector converges to $0$ weakly along a subsequence of density $1$, see, e.g., \cite[Section 16.4]{EFHN}.

Note that every power bounded operator on a reflexive Banach space has relatively weakly compact orbits and hence the above decomposition is valid for e.g.~every contraction on $L^p(X,\mu)$ for $p\in (1,\infty)$.
Moreover, if $T$ is a \emph{Dunford-Schwartz operator} on $L^1(X,\mu)$, i.e., a contraction in $L^1$ which is also a contraction in $L^\infty$, 
then $T$ has relatively weakly compact orbits as well, see Lin, Olsen, Tempelman \cite[Prop.~2.6]{LOT} and Kornfeld, Lin \cite[pp.~226--227]{KL}. 
Note that every Dunford-Schwartz operator is also a contraction on  $L^p(X,\mu)$ for every $p\in(1,\infty)$, see, e.g., \cite[Theorem 8.23]{EFHN}. Thus, the Jacobs-deLeeuw-Glicksberg decomposition is valid for Dunford-Schwartz operators on $L^p(X,\mu)$ for every $p\in[1,\infty)$.

Let  $T$ be a Dunford-Schwartz operator on $(X,\mu)$ (we will write so since $T$ is a contraction on  every $L^p(X,\mu)$, $p\geq 1$). The \emph{(linear) modulus} $|T|$ of $T$ is the unique positive operator on $L^1(X,\mu)$ having  the same $L^1$- and $L^\infty$-norm as $T$ such that $|T^nf|\leq |T|^n |f|$ holds a.e. for every $f\in L^1(X,\mu)$ and every $n\in\mb{N}$. It is again a Dunford-Schwartz operator. For details, see Dunford, Schwartz \cite[p.~672]{DS-book} and Krengel \cite[pp.~159--160]{K-book}. 
Note that for $T$ Dunford-Schwartz, the operators $\lambda T$ for $\lambda\in\mb{T}$ are again Dunford-Schwartz and  have  the same modulus.

For example, every Koopman operator (i.e., the operator induced by a $\mu$-preserving transformation on $X$) is a positive Dunford-Schwartz operator, hence coincides with its modulus, and thus ergodic Koopman operators satisfy the condition $\Fix |T|=\la\mathbf{1}\ra$ appearing in Theorem \ref{thm:main}.  See e.g.~\cite{EFHN} and \cite{petersen:1983} for more information on Koopman operators and an introduction to ergodic theory.

An important property of Dunford-Schwartz operators which we will need is the validity of the pointwise ergodic theorem, i.e., for every $f\in L^1(X,\mu)$ the ergodic averages 
\begin{equation}\label{eq:pet}
\aveN T^nf
\end{equation}
converge a.e.~as $N\to\infty$, see Dunford, Schwartz \cite[p.~675]{DS-book}. 
\begin{Rem}\label{rem:Birk}
Let $T$ be a mean ergodic contraction on $L^1(X,\mu)$
with $\Fix T=\la\mathbf{1}\ra$, and let $f\in L^1(X,\mu)$. Then the $L^1$-limit of (\ref{eq:pet}) equals $c\cdot \mathbf{1}$,
where $c$ is a constant satisfying $|c|\leq \|f\|_1$. Indeed, 
\begin{eqnarray*}
c&=&\|c \mathbf{1}\|_1
=\lim_{N\to\infty}\left\|\aveN T^nf\right\|_1
\leq \lim_{N\to\infty}\aveN \|T^nf\|_1\leq \|f\|_1.
\end{eqnarray*}
In particular, if $T$ is a Dunford-Schwartz operator  with $\Fix T=\la\mathbf{1}\ra$, then the pointwise limit of (\ref{eq:pet}) equals $c\cdot \mathbf{1}$ with $|c|\leq \|f\|_1$.
\end{Rem}

\smallskip

We finally denote by $\ms{P}\subset \ell^\infty$ the set of Bohr almost periodic sequences, i.e., uniform limits of finite linear combinations of sequences of the form $(\lambda^n)$, $\lambda\in \mb{T}$.
The set $\ms{P}$ has the following properties: It is closed in $\ell^\infty$, closed under multiplication, and
is a subclass of  (Weyl) almost periodic sequences $AP(\mb{N})$, i.e., sequences whose orbit under the left shift is relatively compact in $l^\infty$. In fact,  $AP(\mb{N})=\ms{P}\oplus c_0$ holds,  see Bellow, Losert \cite[p. 316]{BL}, corresponding to the Jacobs-deLeeuw-Glicksberg decomposition of $AP(\mb{N})$ induced by the left shift, see, e.g.,~\cite[Theorem I.1.20]{eisner-book}. 

Every element $(a_n)_{n=1}^\infty$ of $AP(\mb{N})$, and hence of  $\ms{P}$, is a good weight for the pointwise ergodic theorem for Dunford-Schwartz operators, i.e., for every Dunford-Schwartz operator $T$ on a probability space and every $f\in L^1(X,\mu)$, the weighted ergodic averages
$$
\aveN a_n T^nf
$$
converge almost everywhere as $N\to\infty$, see {\c{C}}{\"o}mez,  Lin, Olsen \cite[Theorem 2.5]{CLO}.
(Note that also every element of $\ms{N}$ is such a good weight, which is  clear for bounded  functions and follows from the Banach principle for $L^1$-functions. We will however not use it in this paper.)

For more  information and the first part of the following example see, e.g.,~Lin, Olsen, Tempelman \cite{LOT} and Eisner \cite{E}.

\begin{Exam}\label{ex:alm-per}
\begin{enumerate}
\item
If $T$ has relatively weakly compact orbits on a Banach space $E$, $f\in E_r$ and $\varphi\in E'$, then $(\varphi(T^n f))\in \ms{P}$.
\item
Let $(q_k)_{k\in\mb{N}}\in\ell^1$ and $(\gamma_k)\subset \mb{T}$.  Define $(a_n)_{n\in\mb{N}}\subset\ell^\infty$ by 
$$
a_n=
\sum_{k=1}^\infty \gamma_k^n\cdot q_k\quad \forall n\in\mb{N}.
$$
Then $(a_n)\in \ms{P}$.
\end{enumerate}
\end{Exam}

%
%
%
%

\section{A model case}\label{sec:model-case}

Before presenting the proof of the general case, we first explain its ideas on a simpler model where $a=1$, $p=2$ and the decompositions in (A1) are orthogonal.

\begin{Thm}\label{Thm:Main1}
Let 
$(X,\mu)$ be a Borel probability space, $T_0$ be a Dunford-Schwartz operator on $(X,\mu)$,
$H:=L^2(X,\mu)$ and let $H=H_r\oplus H_s$ be the corresponding
Jacobs-deLeeuw-Glicksberg decomposition induced by $T_0$. Let further $A_0\in\mc{L}(H)$ be a bounded operator. For a function $f\in H$,
write $\ms{A}_{f}:=\left\{A_0T_0^nf\left|\right.n\in\mb{N}\right\}$. Suppose that the following holds true:
\begin{center}
For any function $f\in H$ and $\varepsilon>0$, there exists a finite dimensional \\ subspace $\mc{U}=\mc{U}(f,\varepsilon)\subset H$ such that $P_{\mc{U}^\perp}\ms{A}_f\subset B_\epsilon(0,L^\infty(X,\mu))$.\\
\end{center}
Then for any further 
$T_1$ on $(X,\mu)$ with $\Fix |T_1|=\la\mathbf{1}\ra$ 
 we have the following:
\begin{enumerate}
\item for each $f\in H_s$, $\frac{1 }{N}\sum_{n=1}^N |T_1^nA_0T_0^n f|\rightarrow 0$ pointwise a.e.;
\item for each $f\in H_r$, $\frac{1}{N}\sum_{n=1}^N T_1^nA_0T_0^n f$ converges pointwise a.e..
\end{enumerate}
\end{Thm}

\begin{proof}
Let $f\in H$ and $\epsilon>0$ be given. By assumption we have a finite-dimensional subspace $\mc{U}=\mc{U}(f,\varepsilon)\subset H$ such that
$P_{\mc{U}^\perp}\ms{A}_f
\subset B_\varepsilon(0,L^\infty(X,\mu))$. Let $g_1,\ldots,g_k$ be an orthonormal basis in $\mc{U}$. Then we may for each $n\in\mb{N}$ write
\begin{equation}\label{eq:decomp}
A_0T_0^nf=\lambda_{1,n}g_1+\ldots+\lambda_{k,n}g_k+r_n
\end{equation}
for appropriate $\lambda_{j,n}\in\mb{C}$ and $r_n\in \mc{U}^\perp$ with $\|r_n\|_\infty<\varepsilon$. Note that
\[
\lambda_{j,n}= \langle A_0T_0^nf,g_j\rangle=\langle T_0^nf,A_0^*g_j\rangle,
\]
and so $\left|\lambda_{j,n}\right|\leq \|f\|_2\cdot\|A_0^*\|=:c$.

For part (1), assume that $f\in H_s$. Then
\begin{equation}\label{eq:coef}
\lim_{N\to\infty} \frac{1}{N}\sum_{n=1}^N \left| \lambda_{j,n}\right|=\lim_{N\to\infty} \frac{1}{N}\sum_{n=1}^N \left|\langle T_0^nf,A_0^*g_j\rangle\right|=0
\end{equation}
by the definition of $H_s$.
For $\delta:=\varepsilon/ck$ and for each $1\leq j\leq k$ choose a function $\widetilde{g}_j\in L^\infty(X,\mu)$ such that $\|g_j-\widetilde{g}_j\|_1<\delta$.
By Birkhoff's theorem applied to the functions  $g_j-\widetilde{g}_j$ and the operator $|T_1|$,  see Section \ref{sec:prelim} and in particular Remark \ref{rem:Birk}, there exists a set $S_\varepsilon\subset X$ with $\mu(S_\varepsilon)=1$ such that for every $x\in S_\varepsilon$ and every $j\in \{1,\ldots,k\}$ the following conditions hold:
\begin{itemize}
\item $\lim_{N\to\infty}\frac{1}{N}\sum_{n=1}^{N} (|T_1|^n \left|g_j-\widetilde{g}_j\right|)(x)
\leq
\left\|g_j-\widetilde{g}_j\right\|_1$,
\item $|T_1^nr_n(x)|\leq \|r_n\|_\infty$ and $|T_1^n\widetilde{g}_j(x)|\leq \|\widetilde{g}_j\|_\infty$ for every $n\in\mb{N}$.
\end{itemize}
In particular, we have the following inequalities for each $1\leq j\leq k$ and $x\in S_\varepsilon$
\begin{equation}\label{eqn:Birkhoff}
\overline{\lim_{N\to\infty}}\frac{1}{N}\sum_{n=1}^N \left|\lambda_{j,n} T_1^n(g_j-\widetilde{g}_j)\right|(x)\leq \overline{\lim_{N\to\infty}}\frac{1}{N}\sum_{n=1}^{N} c\left(|T_1|^n \left|g_j-\widetilde{g}_j\right|\right)(x)\leq c\delta.
\end{equation}

Consequently, using that $T_1$ is a Koopman operator and hence preserves the $\|\cdot\|_\infty$-norm, we have for each $x\in S_\varepsilon$ using \eqref{eq:coef}
\begin{eqnarray*}
&&\overline{\lim_{N\to\infty}} \frac{1}{N}\sum_{n=1}^N \left|\left(T_1^nA_0T_0^n f\right)(x) \right|=\overline{\lim_{N\to\infty}} \frac{1}{N}\sum_{n=1}^N \left|\left(T_1^nr_n + \sum_{j=1}^{k} \lambda_{j,n}T_1^ng_j\right)(x) \right|\\
&\leq&
\overline{\lim_{N\to\infty}} \frac{1}{N}\sum_{n=1}^N \left|(T_1^nr_n)\right|(x)
+\sum_{j=1}^{k}\overline{\lim_{N\to\infty}} \frac{1}{N}\sum_{n=1}^N  \left|\lambda_{j,n}T_1^ng_j \right|(x)
\\
&\leq&
\overline{\lim_{N\to\infty}}\frac{1}{N}\sum_{n=1}^N  \left\|T_1^nr_n \right\|_\infty
+\sum_{j=1}^{k}\overline{\lim_{N\to\infty}} \frac{1}{N}\sum_{n=1}^N  \left|\lambda_{j,n}T_1^ng_j \right|(x)\\
&\leq&
\varepsilon
+\sum_{j=1}^{k}\overline{\lim_{N\to\infty}} \frac{1}{N}\sum_{n=1}^N  \left|\lambda_{j,n}T_1^n(g_j-\widetilde{g}_j) \right|(x)
+\sum_{j=1}^{k}\overline{\lim_{N\to\infty}} \frac{1}{N}\sum_{n=1}^N  \left|\lambda_{j,n}T_1^n\widetilde{g}_j \right|(x)\\
&\leq&
\varepsilon
+kc\delta
+\sum_{j=1}^{k}\overline{\lim_{N\to\infty}} \frac{1}{N}\sum_{n=1}^N  \left|\lambda_{j,n}T_1^n\widetilde{g}_j \right|(x)
\\
&\leq&
2\varepsilon
+\sum_{j=1}^{k}\overline{\lim_{N\to\infty}} \frac{1}{N}\sum_{n=1}^N  |\lambda_{j,n}|\left\|T_1^n\widetilde{g}_j \right\|_\infty
\leq
2\varepsilon
+\sum_{j=1}^{k}\left\|\widetilde{g}_j \right\|_\infty\overline{\lim_{N\to\infty}} \frac{1}{N}\sum_{n=1}^N  |\lambda_{j,n}|
=2\varepsilon.
\end{eqnarray*}

\noindent Thus for each $x\in\bigcap_{m\in\mb{N}}S_{1/m}=:S$ we have that
\[
\left(\frac{1}{N}\sum_{n=1}^N \left|T_1^nA_0T_0^n f\right|\right)(x)\rightarrow 0.
\]
Since $\mu(S)=1$, we are done.

For part (2), note that eigenfunctions in $H_r$ pertaining to different unimodular eigenvalues are always orthogonal. Take $f\in H_r$ and let $\left\{h_j\right\}_{j=1}^\infty$ be an orthonormal basis in $H_r$ of eigenvectors pertaining to unimodular eigenvalues $\left\{\alpha_j\right\}_{j=1}^\infty$.
(Note that the space $H$ and hence $H_r$ is separable, and we write here an infinite sequence for notational convenience whereas the finite dimensional case can be treated analogously.) Then we can write $f=\sum_{m=1}^\infty d_m h_m$ for some $\ell^2$-sequence $(d_m)_m$ and obtain by the definition of $\lambda_{j,n}$'s in equality (\ref{eq:decomp})
\begin{eqnarray*}
\lambda_{j,n}=\langle T_0^nf,A_0^*g_j\rangle=\big\langle\sum_{m=1}^\infty \alpha_m^nd_m h_m,A_0^*g_j\big\rangle=\sum_{m=1}^\infty \alpha_m^n \left(d_m\langle h_m,A_0^*g_j\rangle\right).
\end{eqnarray*}
By the Cauchy-Schwarz and Bessel inequalities, $(d_m \langle h_m,A_0^*g_j\rangle )_{m=1}^\infty \in l^1$ with the $l^1$-norm bounded by $\|f\|_2\|A_0^*g_j\|_2$.
So for each $1\leq j\leq k$, we have $(\lambda_{j,n})_n\in \ms{P}$, so this sequence is a good weight for the pointwise ergodic theorem for Dunford-Schwartz operators, see 
Example  \ref{ex:alm-per}(2).
In other words, there exists a set $S_\varepsilon\subset X$ with $\mu(S_\varepsilon)=1$ such that for each $1\leq j\leq k$ and all $x\in S_\varepsilon$, the Cesàro means
\begin{eqnarray*}
\frac{1}{N}\sum_{n=1}^N \lambda_{j,n} \left(T_1^n g_j\right)(x)
\end{eqnarray*}
converge. But $\frac{1}{N}\sum_{n=1}^N\|T_1^nr_n\|_\infty\leq\varepsilon$, and so for each $x\in S_\varepsilon$ we have by \eqref{eq:decomp} that
\begin{align*}
&\left|
\overline{\lim}_{N\to\infty}\left(\frac{1}{N}\sum_{n=1}^N T_1^nA_0T_0^n f\right)(x)-\underline{\lim}_{N\to\infty} \left(\frac{1}{N}\sum_{n=1}^N T_1^nA_0T_0^n f\right)(x)
\right|\\
\leq&
\sum_{j=1}^k\left|
\overline{\lim}_{N\to\infty}\left(\frac{1}{N}\sum_{n=1}^N T_1^n  \lambda_{j,n} g_j \right)(x)
-\underline{\lim}_{N\to\infty} \left(\frac{1}{N}\sum_{n=1}^N T_1^n  \lambda_{j,n} g_j\right)(x)
\right|
\\
&+
\left|
\overline{\lim}_{N\to\infty}\left(\frac{1}{N}\sum_{n=1}^N T_1^n r_n \right)(x)
-\underline{\lim}_{N\to\infty} \left(\frac{1}{N}\sum_{n=1}^N T_1^n r_n\right)(x)
\right|
\\
\leq&
\sum_{j=1}^k\left|
\overline{\lim}_{N\to\infty} \frac{1}{N}\sum_{n=1}^N \lambda_{j,n} \left(T_1^n  g_j \right)(x)
-\underline{\lim}_{N\to\infty} \frac{1}{N}\sum_{n=1}^N   \lambda_{j,n}\left(T_1^n g_j\right)(x)
\right|
\\
&+\left|
\overline{\lim}_{N\to\infty}\frac{1}{N}\sum_{n=1}^N \|T_1^n r_n\|_\infty\right|
+\left|\underline{\lim}_{N\to\infty} \frac{1}{N}\sum_{n=1}^N \|T_1^n r_n\|_\infty
\right|
\\
\leq&\, 0+\varepsilon+\varepsilon= 2\varepsilon.
\end{align*}
\noindent Thus for each $x\in\bigcap_{m\in\mb{N}}S_{1/m}=:S$  the limit
\[
\lim_{N\to\infty}\left(\frac{1}{N}\sum_{n=1}^N T_1^nA_0T_0^n f\right)(x),
\]
exists. Since $\mu(S)=1$, the proof is complete.
\end{proof}

%
%

\section{Proof of Theorem \ref{thm:main}}\label{sec:general-case}

Theorem \ref{Thm:Main1}, up to the orthogonality assumption, provides the proof for the simplest case $a=1$ (and $p=2$).

We shall proceed by iterated splitting. 
To avoid cumbersome notations, however, we shall only provide all details for the case $a=2$, and sketch how the ideas carry over to the general case.
We start with  facts concerning the general case; the assumption  $a=2$ will be introduced later on.

Take $f\in E$ and $\epsilon>0$. Then by assumption (A1) we have a finite-dimensional subspace $\mc{U}=\mc{U}(f,\varepsilon/C^{a-1})\subset E$ and a decomposition $E=\mc{U}\oplus \mc{R}$ such that
\[
P_\mc{R} \ms{A}_{0,f}\subset B_{\varepsilon/C^{a-1}}(0,L^\infty(X,\mu)).
\]
Let $g_1,\ldots,g_k$ be a maximal linearly independent set in $\mc{U}$. Then we may for each $n\in\mb{N}$ write
\[
A_0T_0^nf=\lambda_{1,n}g_1+\ldots+\lambda_{k,n}g_k+r_n
\]
for appropriate $\lambda_{j,n}\in\mb{C}$ and $r_n\in \mc{R}$ with $\|r_n\|_\infty<\varepsilon/C^{a-1}$. By the Hahn-Banach theorem we may consider linear forms $\varphi_1,\ldots \varphi_k\in E'$ such that
$$
\varphi_j(g_i)=\delta_{i,j}\quad \mbox{and}\quad \varphi_j|_{\mc{R}}=0\quad  \mbox{ for every } i,j\in\{1,\ldots,k\}.
$$
We then have
\begin{equation}\label{eq:lambda}
\lambda_{j,n}= \varphi_j( A_0T_0^nf)= (A_0^*\varphi_j)(T_0^nf),
\end{equation}
therefore
\begin{equation}\label{eq:lambda2}
\left|\lambda_{j,n}\right|\leq \|f\|_p \cdot\|A_0^*\|_q \max_{j\in\{1,\ldots,k\}}\| \varphi_j\|_q=:c
\end{equation}
for the dual index $q$. Note that $c$ depends on $\varepsilon$. 

Now we have that
\begin{eqnarray*}
&&T_a^nA_{a-1}T^n_{a-1}\ldots A_1T_1^nA_0T_0^n f\\
&=&T_a^nA_{a-1}T^n_{a-1}\ldots A_1T_1^n r_n
+\sum_{j=1}^k T_a^nA_{a-1}T^n_{a-1}\ldots A_1T_1^n \lambda_{j,n}g_j,
\end{eqnarray*}
and we shall investigate the Cesàro convergence of each term separately.

The first term satisfies, by (A2), the inequality
\[
\frac{1}{N}\sum_{n=1}^N |T_a^nA_{a-1}T^n_{a-1}\ldots A_1T_1^n r_n|(x)\leq C^{a-1}\|r_n\|_\infty<\varepsilon
\]
for almost every $x\in X$.

For part (1), assume that $f\in E_{0,s}$. Then (\ref{eq:lambda}) and (\ref{eq:lambda2}) imply  $(\lambda_{j,n})_{n\in\mb{N}}\in\ms{N}$ for each $1\leq j\leq k$.
Fix $1\leq j\leq k$ and consider the term
$$
\frac{1}{N}\sum_{n=1}^N |T_a^nA_{a-1}T^n_{a-1}\ldots A_1T_1^n \lambda_{j,n}g_j|.
$$ As in the proof of Theorem \ref{Thm:Main1}, we may choose a function $\widetilde{g}_j\in L^\infty$ such that $\|g_j-\widetilde{g}_j\|_1\leq\|g_j-\widetilde{g}_j\|_p<\varepsilon/ck$.
Then
\begin{eqnarray*}
&&\frac{1}{N}\sum_{n=1}^N |T_a^nA_{a-1}T^n_{a-1}\ldots A_1T_1^n \lambda_{j,n} g_j|\\
&\leq&
\frac{1}{N}\sum_{n=1}^N |T_a^nA_{a-1}T^n_{a-1}\ldots A_1T_1^n \lambda_{j,n}(g_j-\widetilde{g}_j)|
+
\frac{1}{N}\sum_{n=1}^N |T_a^nA_{a-1}T^n_{a-1}\ldots A_1T_1^n \lambda_{j,n}\widetilde{g}_j|.
\end{eqnarray*}

Since $(\lambda_{j,n})_{n\in\mb{N}}\in\ms{N}$, the second term satisfies by (A2)
\[
\frac{1}{N}\sum_{n=1}^N |T_a^nA_{a-1}T^n_{a-1}\ldots A_1T_1^n \lambda_{j,n}\widetilde{g}_j|(x)\leq C^{a-1}\|\widetilde{g}_j\|_\infty\cdot\frac{1}{N}\sum_{n=1}^N |\lambda_{j,n}|\to 0
\]
for almost every $x\in X$.

It now remains to treat the first term. Again using our assumption (A1), there exists a finite dimensional subspace $\mc{U}_j=\mc{U}(g_j-\widetilde{g}_j,\varepsilon/kC^{a-2})\subset E$ and a decomposition $E=\mc{U}_j\oplus\mc{R}_j$ such that $P_{\mc{R}_j}\ms{A}_{1,g_j-\widetilde{g}_j}\subset B_{\varepsilon/kC^{a-2}}(0,L^\infty(X,\mu))$. Let $g_{1,j},g_{2,j},\ldots,g_{k_j,j}$ be a maximal linearly independent set in $\mc{U}_j$ and choose $\varphi_{1,j},\ldots, \varphi_{k_j,j}\in E'$ to have the property
$$
\varphi_{i,j}(g_{l,j})=\delta_{i,l}\quad \mbox{and}\quad \varphi_{i,j}|_{\mc{R}_j}=0\quad  \mbox{ for every } i,l\in\{1,\ldots,k_j\}
$$
which is possible by the Hahn-Banach theorem.
Then, for each $n\in\mb{N}$, we write
\[
A_1T_1^n (g_j-\widetilde{g}_j)=\lambda_{1,j,n}g_{1,j}+\ldots+\lambda_{k_j,j,n}g_{k_j,j}+r_{j,n}
\]
for appropriate $\lambda_{i,j,n}\in\mb{C}$ ($1\leq i\leq k_j$) and $r_{j,n}\in \mc{R}_j$ with $\|r_{j,n}\|_\infty<\varepsilon/kC^{a-2}$ and obtain
\[
\lambda_{i,j,n}= \varphi_{i,j}( A_1T_1^n(g_j-\widetilde{g}_j))=(A_1^*\varphi_{i,j})(T_1^n(g_j-\widetilde{g}_j)).
\]
It follows that
$$
\left|\lambda_{i,j,n}\right|\leq \|g_j-\widetilde{g}_j\|_p\cdot\|A_1^*\|_q \cdot\max_{i\in\{1,\ldots,k_j\}}\|\varphi_{i,j}\|_q=:c_j
$$
for the dual index $q$.
Now for each $1\leq i\leq k_j$ choose a function $\widetilde{g}_{i,j}\in L^\infty$ such that $\|g_{i,j}-\widetilde{g}_{i,j}\|_1\leq\|g_{i,j}-\widetilde{g}_{i,j}\|_p<\varepsilon/(cc_jk_jk)$.

Thus we write
\begin{eqnarray*}
&&T_a^nA_{a-1}T^n_{a-1}\ldots A_1T_1^n \lambda_{j,n}(g_j-\widetilde{g}_j)\\
&=&\sum_{i=1}^{k_j} T_a^n\ldots A_2T_2^n \lambda_{j,n}\lambda_{i,j,n}(g_{i,j}-\widetilde{g}_{i,j})
+\sum_{i=1}^{k_j} T_a^n\ldots A_2T_2^n \lambda_{j,n}\lambda_{i,j,n}\widetilde{g}_{i,j}\\
&+&T_a^n\ldots A_2T_2^n \lambda_{j,n} r_{j,n}.
\end{eqnarray*}
When taking the Cesàro averages over $n$ of the absolute values, the contribution of last term tends to 0 for almost every $x\in X$, since $(\lambda_{j,n})_{n\in\mb{N}}\in\ms{N}$. The contribution of the second sum also tends to zero almost everywhere, due to $(\lambda_{j,n}\lambda_{i,j,n})_{n\in\mb{N}}\in\ms{N}$ (as $\ms{N}$ is closed under multiplication by bounded sequences) and by $\widetilde{g}_{i,j}\in L^\infty(X,\mu)$ and (A2).

Now, when $a=2$, using the fact that $T_2$ is a 
Dunford-Schwartz operator, the contribution of the first sum is bounded by
\[
\frac{1}{N}\sum_{n=1}^N \sum_{i=1}^{k_j} |T_2|^n |\lambda_{j,n}\lambda_{i,j,n}(g_{i,j}-\widetilde{g}_{i,j})|.
\]
In this case we can use the boundedness of the $\lambda_*$ sequences and the pointwise ergodic theorem for  Dunford-Schwartz operators (cf. Remark \ref{rem:Birk} and equation (\ref{eqn:Birkhoff}) from the proof of Theorem \ref{Thm:Main1}) to see that there is a set $S_{j,\varepsilon}$ with $\mu(S_{j,\varepsilon})=1$ such that for all $x\in S_{j,\varepsilon}$ this contribution has a limes superior not exceeding $c\cdot c_j\cdot (\varepsilon/cc_jk_jk)=\varepsilon/k_jk$. All other contributions discussed above have limit zero.

Summing over all $1\leq i\leq k_j$ and then $1\leq j\leq k$, we have for each $x\in\cap_{j=1}^k S_{j,\varepsilon}=:S_\varepsilon$ that

\begin{eqnarray*}
&&\overline{\lim_{N\to\infty}} \frac{1}{N}\sum_{n=1}^N \left|\left(T_2^nA_1T_1^nA_0T_0^n f\right)(x) \right|\\
&=&
\overline{\lim_{N\to\infty}} \frac{1}{N}\sum_{n=1}^N
 \left|
\left(
T_2^nA_1T_1^nr_n
+ \sum_{j=1}^{k} T_2^n \lambda_{j,n}r_{j,n}
+ \sum_{j=1}^k\sum_{i=1}^{k_j} T_2^n \lambda_{j,n}\lambda_{i,j,n}\widetilde{g}_{i,j}
\right.
\right.\\
&&+
\left.
\left.
\sum_{j=1}^k\sum_{i=1}^{k_j} T_2^n \lambda_{j,n}\lambda_{i,j,n}(g_{i,j}-\widetilde{g}_{i,j})
\right)(x)
\right|
\\
&\leq&
\overline{\lim_{N\to\infty}} \frac{1}{N}\sum_{n=1}^N
 \left|
\left(
T_2^nA_1T_1^nr_n
\right)(x)
\right|
+\sum_{j=1}^{k}\overline{\lim_{N\to\infty}} \frac{1}{N}\sum_{n=1}^N
 \left|
\left(
T_2^n \lambda_{j,n}r_{j,n}
\right)(x)
\right|\\
&&+\sum_{j=1}^k\sum_{i=1}^{k_j}\overline{\lim_{N\to\infty}} \frac{1}{N}\sum_{n=1}^N
 \left|
\left(
T_2^n \lambda_{j,n}\lambda_{i,j,n}\widetilde{g}_{i,j}
\right)(x)
\right|\\
&&+\sum_{j=1}^k\sum_{i=1}^{k_j}\overline{\lim_{N\to\infty}} \frac{1}{N}\sum_{n=1}^N
 \left|
\left(
\sum_{j=1}^k\sum_{i=1}^{k_j} T_2^n \lambda_{j,n}\lambda_{i,j,n}(g_{i,j}-\widetilde{g}_{i,j})
\right)(x)
\right|\\
&\leq&\varepsilon+\sum_{j=1}^{k}0+\sum_{j=1}^k\sum_{i=1}^{k_j}0+\sum_{j=1}^k\sum_{i=1}^{k_j} \varepsilon/k_jk=2\varepsilon.
\end{eqnarray*}
\noindent Thus for each $x\in\bigcap_{m\in\mb{N}}S_{1/m}=:S$ we have that
\[
\left(\frac{1}{N}\sum_{n=1}^N \left|T_2^nA_1T_1^nA_0T_0^n f\right|\right)(x)\rightarrow 0,
\]
and, since $\mu(S)=1$, we are done.

 If $a>2$, then we from here iterate the following for each operator pair $A_zT_z^n$ ($2\leq z\leq a-1$).\\
We consider the last untreated sum from the previous step, the one containing the contribution arising from the functions $A_zT_z^n(g_*-\widetilde{g}_*)$. Using assumption $(A1)$, we split each such function further into a linear combination of finitely many functions $g_{\ell,*}\in E$ and a remainder term $r_{*,n}\in L^\infty$. The new coefficient sequences $\lambda_{\ell,*,n}$ will also lie in $\ms{N}$, hence the contribution of remainder terms to the Cesàro means will be zero. Then, as seen for $g_j$, we split each of the $g_{\ell,*}$ into an essentially bounded part $\widetilde{g}_{\ell,*}\in L^\infty$ and a remainder small in $L^1$. In the Cesàro means, using that all coefficient sequences lie in $\ms{N}$ and by assumption $(A2)$, the terms with $\widetilde{g}_{\ell,*}$ all have zero contribution, and so we are left with the functions $g_{\ell,*}-\widetilde{g}_{\ell,*}$, from where we continue the iteration.\\
At the end, we reach $T_a^n$, applied to functions $g_*-\widetilde{g}_*$ (sufficiently small in $L^1$) with coefficients being products of $\lambda$-s. At this point, as detailed for $T_2^n$ when we assumed $a=2$, we use the boundedness of the coefficient sequences, and apply Birkhoff's pointwise ergodic theorem for Dunford-Schwartz operators to $|g_*-\widetilde{g}_*|$ to obtain a contribution to the limsup of the Cesàro means that adds up to $2\varepsilon$ over all -- finitely many -- multiindices $*$.

\vspace{0.15cm}

For part (2), assume $p=2$, write $H:=L^2(X,\mu)$ and note that eigenfunctions in $H_{0,r}$ pertaining to different unimodular eigenvalues are orthogonal. For notational convenience we again assume that $H_{0,r}$ is infinite-dimensional, whereas the finite dimensional case can be treated analogously. Take $f\in H_{0,r}$ and let $\left\{h_j\right\}_{j=1}^\infty$ be an orthonormal basis in $H_{0,r}$ of eigenvectors of $T_0$ pertaining to unimodular eigenvalues $\left\{\alpha_j\right\}_{j=1}^\infty$. Then we can write $f=\sum_{m=1}^\infty d_m h_m$ for some $\ell^2$-sequence $(d_m)_m$ and obtain
\begin{eqnarray*}
\lambda_{j,n}
 =\langle T_0^nf,A_0^*\varphi_j\rangle=\big\langle\sum_{m=1}^\infty \alpha_m^nd_m h_m,A_0^*\varphi_j\big\rangle=\sum_{m=1}^\infty \alpha_m^n \left(d_m\langle h_m,A_0^*\varphi_j\rangle\right).
\end{eqnarray*}
So for each $1\leq j\leq k$ we have that
$(\lambda_{j,n})_n\in\ms{P}$
since $\left(d_m\langle h_m,A_0^*\varphi_j\rangle\right)\in l^1$ by the Cauchy-Schwarz inequality.

For each $1\leq j\leq k$, we may split $g_j$ into the (almost weakly) stable and the reversible part with respect to $T_1$, i.e. $g_j=g_j^s+g_j^r$ with $g_j^s\in H_{1,s}$ and $g_j^r\in H_{1,r}$. Then we have
\begin{eqnarray*}
\sum_{j=1}^k T_a^nA_{a-1}T^n_{a-1}\ldots A_1T_1^n \lambda_{j,n}g_j
=\sum_{j=1}^{k_j} T_a^n\ldots A_1T_1^n \lambda_{j,n}g_{j}^r
+\sum_{j=1}^{k} T_a^n\ldots A_1T_1^n \lambda_{j,n} g_j^s.
\end{eqnarray*}

We first look at the contribution of the second sum to the Cesàro averages. Observe that
\begin{eqnarray*}
&&\left|
\frac{1}{N}\sum_{n=1}^N \sum_{j=1}^{k} T_a^n\ldots A_1T_1^n \lambda_{j,n} g_j^s
\right|(x)
\leq
\frac{1}{N}\sum_{n=1}^N \sum_{j=1}^{k}
\left|
T_a^n\ldots A_1T_1^n \lambda_{j,n} g_j^s
\right|(x)\\
&\leq&
\frac{1}{N}\sum_{n=1}^N \sum_{j=1}^{k}
c \left|
T_a^n\ldots A_1T_1^n  g_j^s
\right|(x)
=
 \sum_{j=1}^{k}
\left(
\frac{1}{N}\sum_{n=1}^N
c \left|
T_a^n\ldots A_1T_1^n  g_j^s
\right|(x)
\right)\to 0
\end{eqnarray*}
for almost all $x\in X$, using part (1) applied to $(a-1)$ pairs $A_iT_i^n$.

We now turn our attention to the first sum, involving the reversible parts $g_j^r$.
For each $1\leq j\leq k$ there exists a finite dimensional subspace
\[
\mc{U}_j=\mc{U}(g_j^r,\varepsilon/C^{a-2})\subset H
\]
and a decomposition $E=\mc{U}_j\oplus\mc{R}_j$
such that
\[
P_{\mc{R}_j}\ms{A}_{1,g_j^r}\subset B_{\varepsilon/C^{a-2}}(0,L^\infty(X,\mu)).
\]
 Let $g_{1,j},g_{2,j},\ldots,g_{k_j,j}$ be an orthonormal basis in $\mc{U}_j$. Then we write for each $n\in\mb{N}$
\[
A_1T_1^n (g_j^r)=\lambda_{1,j,n}g_{1,j}+\ldots+\lambda_{k_j,j,n}g_{k_j,j}+r_{j,n}
\]
for appropriate $\lambda_{i,j,n}\in\mb{C}$ ($1\leq i\leq k_j$) and $r_{j,n}\in \mc{R}_j
$ with $\|r_{j,n}\|_\infty<\varepsilon/C^{a-2}$ and observe
\[
\lambda_{i,j,n}= \langle A_1T_1^n g_j^r,\varphi_{i,j}\rangle=\langle T_1^n g_j^r,A_1^*\varphi_{i,j}\rangle,
\]
where as above each $\varphi_{i,j}$ is orthogonal to $\mc{R}_j$ and $\langle g_{l,j},\varphi_{i,j}\rangle=\delta_{l,i}$.
(Note that if $\mc{R}_j\perp\mc{U}_j$, then we can choose
$\varphi_{i,j}:=g_{i,j}$.)
So
$$
\left|\lambda_{i,j,n}\right|\leq \|g_j^r\|_2\cdot\|A_1^*\|\max\{\|\varphi_{i,j}\|_2,\, i=1,\ldots,k_j\}=:c_j
$$
and $(\lambda_{i,j,n})_{n\in\mb{N}}\in\ms{P}$ by Example \ref{ex:alm-per}.

Thus for each $1\leq j\leq k$ we have for almost every $x\in X$
\begin{eqnarray*}
&&\left|\overline{\lim}_{N\to\infty}
\left(
\frac{1}{N}\sum_{n=1}^N  T_a^n\ldots A_2T_2^n \lambda_{j,n} g_j^r
\right)(x)
\right.
-
\left.
\underline{\lim}_{N\to\infty
}
\left(
\frac{1}{N}\sum_{n=1}^N  T_a^n\ldots A_2T_2^n \lambda_{j,n} g_j^r
\right)(x)
\right|\\
&\leq&
\left|\overline{\lim}_{N\to\infty}
\left(
\frac{1}{N}\sum_{n=1}^N  T_a^n\ldots A_2T_2^n \lambda_{j,n} r_{j,n}
\right)(x)
\right.
-
\left.
\underline{\lim}_{N\to\infty}
\left(
\frac{1}{N}\sum_{n=1}^N  T_a^n\ldots A_2T_2^n \lambda_{j,n} r_{j,n}
\right)(x)
\right|
\\
&&+
\sum_{i=1}^{k_j}
\left|\overline{\lim}_{N\to\infty}
\left(
\frac{1}{N}\sum_{n=1}^N  T_a^n\ldots A_2T_2^n \lambda_{j,n}\lambda_{i,j,n} g_{i,j}
\right)(x)
\right.
\\
&&
-
\left.
\underline{\lim}_{N\to\infty}
\left(
\frac{1}{N}\sum_{n=1}^N  T_a^n\ldots A_2T_2^n \lambda_{j,n}\lambda_{i,j,n} g_{i,j}
\right)(x)
\right|^.
\end{eqnarray*}
The first difference on the right hand side is bounded by $2C^{a-2}\|r_{j,n}\|_\infty\leq 2\varepsilon$.

If now $a=2$, then the sum at the end consists of terms of the form
\[
\left|
\overline{\lim}_{N\to\infty}\left(\frac{1}{N}\sum_{n=1}^N T_2^n \lambda_{j,n}\lambda_{i,j,n} g_{i,j}\right)(x)
-\underline{\lim}_{N\to\infty} \left(\frac{1}{N}\sum_{n=1}^N T_2^n\lambda_{j,n}\lambda_{i,j,n} g_{i,j} \right)(x)
\right|.
\]
Note that $(\lambda_{j,n})_n\in\ms{P}$ and $(\lambda_{i,j,n})_n\in\ms{P}$ implies $(\lambda_{j,n}\lambda_{i,j,n})_n\in\ms{P}$, and since elements in $\ms{P}$ are good weights for the pointwise ergodic theorem for Dunford-Schwartz operators, this absolute value is zero for almost all $x$.

Summing up, we obtain
\begin{eqnarray*}
&&\left|
\overline{\lim}_{N\to\infty}
\left(
\frac{1}{N}\sum_{n=1}^N T_2^nA_1T_1^nA_0T_0^n f
\right)(x)
-
\underline{\lim}_{N\to\infty}
\left(
\frac{1}{N}\sum_{n=1}^N T_2^nA_1T_1^nA_0T_0^n f
\right)(x)
\right|\\
&\leq&
\left|
\overline{\lim}_{N\to\infty}
\left(
\frac{1}{N}\sum_{n=1}^N T_2^nA_1T_1^n r_n
\right)(x)
-
\underline{\lim}_{N\to\infty}
\left(
\frac{1}{N}\sum_{n=1}^N T_2^nA_1T_1^n r_n
\right)(x)
\right|\\
&&+\sum_{j=1}^k
\left|
\overline{\lim}_{N\to\infty}
\left(
\frac{1}{N}\sum_{n=1}^N T_2^nA_1T_1^n  \lambda_{j,n}g_j
\right)(x)
-
\underline{\lim}_{N\to\infty}
\left(
\frac{1}{N}\sum_{n=1}^N T_2^nA_1T_1^n \lambda_{j,n}g_j
\right)(x)
\right|\\
&\leq&
2\varepsilon
+\sum_{j=1}^k
\left|
\overline{\lim}_{N\to\infty}
\left(
\frac{1}{N}\sum_{n=1}^N T_2^nA_1T_1^n \lambda_{j,n}g_j^s
\right)(x)
-
\underline{\lim}_{N\to\infty}
\left(
\frac{1}{N}\sum_{n=1}^N T_2^nA_1T_1^n \lambda_{j,n}g_j^s
\right)(x)
\right|\\
&&+\sum_{j=1}^k
\left|
\overline{\lim}_{N\to\infty}
\left(
\frac{1}{N}\sum_{n=1}^N T_2^nA_1T_1^n \lambda_{j,n}g_j^r
\right)(x)
-
\underline{\lim}_{N\to\infty}
\left(
\frac{1}{N}\sum_{n=1}^N T_2^nA_1T_1^n \lambda_{j,n}g_j^r
\right)(x)
\right|\\
&=&
2\varepsilon
+\sum_{j=1}^k
\left|
\overline{\lim}_{N\to\infty}
\left(
\frac{1}{N}\sum_{n=1}^N T_2^nA_1T_1^n \lambda_{j,n}g_j^r
\right)(x)
-
\underline{\lim}_{N\to\infty}
\left(
\frac{1}{N}\sum_{n=1}^N T_2^nA_1T_1^n \lambda_{j,n}g_j^r
\right)(x)
\right|\\
&\leq& 2\varepsilon+2\varepsilon=4\varepsilon
\end{eqnarray*}
for all $x\in S_\varepsilon$ for some appropriate $S_\varepsilon\subset X$ with $\mu(S_\varepsilon)=1$.
\noindent Thus for each $x\in\bigcap_{m\in\mb{N}}S_{1/m}=:S$ we have that
\[
\left|
\overline{\lim}_{N\to\infty}
\left(
\frac{1}{N}\sum_{n=1}^N T_2^nA_1T_1^nA_0T_0^n f
\right)(x)
-
\underline{\lim}_{N\to\infty}
\left(
\frac{1}{N}\sum_{n=1}^N T_2^nA_1T_1^nA_0T_0^n f
\right)(x)
\right|\rightarrow 0.
\]
Since $\mu(S)=1$, this completes the case $a=2$.

For the case when $a>2$, for each pair $(i,j)$, we split the function $g_{i,j}$ into its stable and reversible part with respect to $T_2$, and apply the above arguments until we reach the last operator $T_a$. In each split, the stable parts $g_*^s$ will contribute with a pointwise almost everywhere zero Cesàro average limit each, and the remainder parts $r_*$ have a total spread between the limes superior and the limes inferior bounded by $2\varepsilon$. The last reversible parts $T_ag^r_*$ converge pointwise almost everywhere since the sequence of weights is a product of elements of $\ms{P}$, and hence an element of $\ms{P}$ itself, being a good sequence of weights.

In total, we obtain that
\begin{eqnarray*}
&&\left|
\overline{\lim}_{N\to\infty}
\left(
\frac{1}{N}\sum_{n=1}^N T_a^nA_{a-1}T^n_{a-1}\ldots A_1T_1^nA_0T_0^n f
\right)(x)
\right.\\
&&-
\left.
\underline{\lim}_{N\to\infty}
\left(
\frac{1}{N}\sum_{n=1}^N T_a^nA_{a-1}T^n_{a-1}\ldots A_1T_1^nA_0T_0^n f
\right)(x)
\right|\\
&\leq&a\cdot 2\varepsilon
\end{eqnarray*}
for all $x$ outside of a nullset, completing the proof. (Recall that the form (\ref{eq:lim}) of the limit  is the same as in the norm case and follows from  \cite[Theorem 3]{EKK}.)

\begin{Rem}
For eigenfunctions  $f\in L^p(X,\mu)$ of $T_0$, the averages
$$
\aveN T_1^n A_0 T_0^n f
$$
converge a.e. for every operator $A_0$ on $E:=L^p(X,\mu)$, $p\in[1,\infty)$. Indeed, if $T_0f=\lambda f$ for some $\lambda\in \mb{T}$, then the above averages take the form
$$
\aveN (\lambda T_1)^n A_0f.
$$
Since $\lambda T_1$ is again a Dunford-Schwartz operator, a.e.~convergence of the above 
averages follows from the pointwise ergodic theorem. However, due to the lack of a Banach principle, it is not clear how to conclude convergence for arbitrary $f\in E_r$, $E_r$ being the reversible part of $E$ corresponding to $T_0$, for $p\neq 2$.
\end{Rem}

%
%
%
%

\section{Examples and a continuous analogue}\label{sec:ex}

\subsection{Examples: powers of the Volterra operator}

Consider on $H:=L^2([0,1])$ the Volterra operator $V$ given by
\[
(Vf)(x):=\int_{0}^x f(t) \mr{dt}.
\]
We first check that $V$ can be written as a sum of three operators which satisfy conditions (A1) and (A2) of Theorem \ref{thm:main} for any  Dunford-Schwartz 
operators.

With the orthonormal base $e_m(x):=e^{2\pi imx}$, we have for $0\neq m\in\mb{Z}$
\[(Ve_m)(x)=\int_0^x e^{2\pi imt} \mr{dt}=\frac{1}{2\pi im}\left(e_m(x)-1\right),
\]
and thus for an $f\in H$ with the base decomposition $f=\sum_{m\in\mb{Z}} c_me_m$ (where $(c_m)_m$ is an $\ell^2$-sequence) we may write
\[
(Vf)(x)=\left(\frac{1}{2\pi i}\sum_{0\neq m\in\mb{Z}}c_m\frac{e_m(x)-1}{m} \right)+c_0x.
\]

Consider now the decomposition of the Volterra operator into the sum $V=V_1+V_2+V_3$ with
\begin{eqnarray*}
V_1f:=c_0\cdot J
,\quad
V_2f:= -\frac{1}{2\pi i}\sum_{0\neq m\in\mb{Z}}\frac{c_m}{m}e_0,\quad
V_3f:=\frac{1}{2\pi i}\sum_{0\neq m\in\mb{Z}}\frac{c_m}{m}e_m,
\end{eqnarray*}
where $f=\sum_{m=-\infty}^\infty c_m e_m$ and $J(x)=x$.

The operators $V_1$ and $V_2$ both have one-dimensional range and are bounded with respect to the $L^\infty$-norm. Indeed, the last assertion for $V_2$ follows from
$$
\|V_2f\|_\infty\leq \frac{1}{2\pi}\sum_{0\neq m\in\mb{Z}}\frac{|c_m|}{|m|} \leq \frac{1}{2\pi}\left(\sum_{0\neq m\in\mb{Z}} |c_m|^2 \sum_{0\neq m\in\mb{Z}} \frac{1}{m^2} \right)^{1/2} \leq \frac{\|f\|_2}{2\sqrt{3}}\leq \frac{\|f\|_\infty}{2\sqrt{3}}.
$$
Thus, assumptions (A1) and (A2) are satisfied for both $V_1$ and $V_2$ as well as any choice of  Dunford-Schwartz  
operators $T_j$.
It remains to show that the same holds for $V_3$, too.

Assumption (A2) is satisfied for $V_3$ and any  Dunford-Schwartz
operator by the same calculation as for $V_2$. To show (A1),
let $\varepsilon>0$ and $f\in H$ with $\|f\|_2\leq 1$ be fixed. We may choose $M\in\mb{N}$ such that $\sum_{|m|\geq M}\frac{1}{m^2}<4\pi^2\varepsilon^2$. Then with the decomposition $V_3f=g_1+g_2$, where
\[
g_1:=\frac{1}{2\pi i}\sum_{0< |m|<M}\frac{c_m}{m}e_m,
\]
we have by the Cauchy-Schwarz and Bessel's inequalities
\begin{eqnarray*}
\|g_2\|_\infty \leq \frac{1}{2\pi}\sum_{|m|\geq M} \frac{|c_m|}{|m|}
&\leq& \frac{1}{2\pi}\left(\sum_{|m|\geq M}|c_m|^2\right)^{1/2} \left(\sum_{|m|\geq M}\frac{1}{m^2}\right)^{1/2}\\
&\leq& \frac{1}{2\pi}\|f\|_2 \left(\sum_{|m|\geq M}\frac{1}{m^2}\right)^{1/2}
<\varepsilon.
\end{eqnarray*}
Thus taking
$\mc{U}:=\mr{span}\left\{e_m:|m|<M\right\}$ and $\mc{R}:=\mr{span}\{e_m:|m|\geq M\}$ we have the desired (orthogonal) decomposition in condition (A1) for the operator $V_3$ and any  Dunford-Schwartz operators.

Analogously, for every $k\in \mb{N}$ the operator $V^k=(V_1+V_2+V_3)^k$ decomposes into a finite sum of one-dimensional operators (each term containing at least one $V_1$ or $V_2$) which are bounded with respect to the $L^\infty$-norm and the operator $V_3^k$ of the form $V_3^kf=(2\pi i)^{-k} \sum_{0\neq m\in\mb{Z}}\frac{c_m}{m^k}e_m$. The properties (A1) and (A2) for $V_3^k$ follow analogously to the above calculations for $V_3$. Hence, (A1) and (A2) are satisfied for $V^k$ and any choice of  Dunford-Schwartz
operators $T_j$.

Putting everything together, we obtain for any choice of  Dunford-Schwartz
operators $T_0,\ldots,T_a$  with $\Fix |T_1|=\ldots=\Fix |T_a|=\la\textbf{1}\ra$ and for every $k_0,\ldots,k_{a-1}\in \mb{N}$ that
\begin{enumerate}
\item for each $f\in H_{0,s}$, $\lim_{N\to\infty}\frac{1}{N}\sum_{n=1}^N |T_a^nV^{k_{a-1}}T^n_{a-1}\ldots V^{k_1}T_1^nV^{k_0}T_0^n f|= 0$ pointwise a.e.;
\item for each $f\in H_{0,r}$, $\frac{1}{N}\sum_{n=1}^N T_a^nV^{k_{a-1}}T^n_{a-1}\ldots V^{k_1}T_1^nV^{k_0}T_0^n f$ converges pointwise a.e..
\end{enumerate}

\subsection{Continuous version}

In this section we consider strongly continuous (shortly: $C_0$-) semigroups $(T_j(t))_{t\in [0,\infty)}$ instead of discrete semigroups $(T_j^n)_{n=0}^\infty$, $j\in\{0,\ldots,a\}$. 

Let $T(\cdot):=(T(t))_{t\in [0,\infty)}$ be a $C_0$-semigroup of Dunford-Schwartz operators on $L^1(X,\mu)$. Then, since the unit ball in $L^\infty(X,\mu)$ is invariant under the semigroup, $T(\cdot)$ is by the standard approximation argument automatically  a $C_0$-semigroup  (of contractions) on $L^p(X,\mu)$ for every $\infty>p\geq 1$  (note that the reverse implication also holds). Moreover, for every $f\in L^1(X,\mu)$ the function $(T(\cdot)f)(x)$ is Lebesgue integrable over finite intervals in $[0,\infty)$ for almost every $x\in X$ by Fubini's theo\-rem, see, e.g., Sato \cite[p.~3]{S}. Analogously, for $C_0$-semigroups $T_0(\cdot),\ldots,T_a(\cdot)$ on $E:=L^p(X,\mu)$, operators $A_0,\ldots,A_{a-1}\in\mc{L}(E)$ and $f\in E$, the function 
$$
(T_a(\cdot)A_{a-1}T_{a-1}(\cdot)\ldots A_1T_1(\cdot)A_0T_0(\cdot) f)(x)
$$
is  Lebesgue integrable over finite intervals in $[0,\infty)$ for almost every $x\in X$.

The pointwise ergodic theorem extends to every strongly measurable semigroup $T(\cdot)$ of Dunford-Schwartz operators, see Dunford, Schwartz \cite[pp. 694, 708]{DS-book}. Moreover, as in Remark \ref{rem:Birk}, $\cap_{t>0}\Fix T(t)=\la \textbf{1}\ra$ implies that 
$$
\lim_{\mc{T}\to\infty}\frac{1}{\mc{T}}\int_{0}^\mc{T} T(t)f\, dt=c \cdot\textbf{1}
$$
with $|c|\leq \|f\|_1.$ 
Furhermore, a natural modification of Lin, Olsen, Tempelman \cite[Proof of Prop. 2.6]{LOT} shows that every $C_0$-semigroup of Dunford-Schwartz operators has relatively weakly compact orbits in $L^1(X,\mu)$. Thus, the continuous version of the Jacobs-deLeeuw-Glicksberg decomposition (see e.g.~\cite[Theorem III.5.7]{eisner-book}) is valid for such semigroups.

We also need a continuous analogue of the concept of the modulus. By e.g.~Kipnis \cite{Ki} or Kubokawa \cite{Ku}, for a  $C_0$-semigroup $T(\cdot)$ of contractions there exists a minimal  $C_0$-semigroup dominating $T(\cdot)$ which is also contractive. We denote this positive semigroup by $|T|(\cdot)$ and refer to Becker, Greiner \cite{BG} for related results. (Note that  $|T|(t)\neq |T(t)|$ in general.) Of course, $|T|(\cdot)=T(\cdot)$ for positive semigroups.  Moreover, the construction in \cite[pp. 372-3]{Ki} implies that if $T(\cdot)$ consists Dunford-Schwartz operators then so does $|T|(\cdot)$.

Analogously to the proof of Theorem \ref{thm:main} we obtain the following continuous version of Theorem \ref{thm:main}. (Cf. Bergelson, Leibman, Moreira \cite{BLM} for an abstract method  of transferring discrete results into continuous ones.)

\begin{Thm}\label{thm:main-cont}
For $a\in\mb{N}$, let
$(T_0(t))_{t\in [0,\infty)}$,$(T_1(t))_{t\in [0,\infty)}$, $\ldots $, $(T_a(t))_{t\in [0,\infty)}$ be  
$C_0$-semigroups of Dunford-Schwartz operators on $L^1(X,\mu)$ of a Borel probability space $(X,\mu)$, with 
$$\cap_{t>0}\Fix|T_1|(t)=\ldots=\cap_{t>0}\Fix|T_a|(t)=\la \textbf{1}\ra.$$
For $p\in[1,\infty)$  and $E:=L^p(X,\mu)$, let $E=E_{0,r}\oplus E_{0,s}$ be the Jacobs-deLeeuw-Glicksberg decomposition corresponding to $T_0(\cdot)$.
Let further $A_j\in\mc{L}(E)$ $(0\leq j< a)$ be bounded operators. For a function $f\in E$ and an index $0\leq j< a$, write $\ms{A}_{j,f}:=\left\{A_jT_j(t)f\left|\right.t\in(0,\infty)\right\}$. Suppose that the following conditions hold:
\begin{itemize}
\item[(A1)]\emph{(Twisted compactness)}
For any function $f\in E$, index $0\leq j<a$ and $\varepsilon>0$, there exists a decomposition $E=\mc{U}\oplus \mc{R}$ with $\dim \mc{U}<\infty$
such that $P_\mc{R}\ms{A}_{j,f}
\subset B_\varepsilon(0,L^\infty(X,\mu))$.
\item[(A2)]\emph{(Joint $L^\infty$-boundedness)}
There exists a constant $C>0$ such that we have
\[\{A_jT_j(t)|\,t\in (0,\infty),1\leq j< a\}\subset B_C(0,\mc{L}(L^\infty(X,\mu)).
\]
\end{itemize}
Then 
\begin{enumerate}
\item for each $f\in E_{0,s}$,
$$
\lim_{\mc{T}\to\infty}\frac{1}{\mc{T}}\int_{0}^\mc{T} |T_a(t)A_{a-1}T_{a-1}(t)\ldots A_1T_1(t)A_0T_0(t) f|\,dt=0 \quad \text{ pointwise a.e.};
$$
\item if $p=2$, then for each $f\in E_{0,r}$,
$$
\frac{1}{\mc{T}}\int_{0}^\mc{T} T_a(t)A_{a-1}T_{a-1}(t)\ldots A_1T_1(t)A_0T_0(t) f\,dt
$$
 converges pointwise a.e..
\end{enumerate}
\end{Thm}
\begin{Rem}
If for some $j\in\{1,\ldots,a\}$ the semigroup $T_j(\cdot)$ consists of positive operators, then one can replace  the condition $\cap_{t>0}\Fix|T_j|(t)=\la \textbf{1}\ra$  by $\ker(G_j)=\la \textbf{1}\ra$ for the generator $G_j$ of $T_j(\cdot)$, see, e.g., Engel, Nagel \cite[Cor.~IV.3.8]{EN-book}.  Moreover, this condition for the semigroup induced by a measure preserving flow is equivalent to the ergodicity of the flow.
\end{Rem}
Note that the examples of powers of the Volterra operator discussed above are valid in the continuous setting as well.

\vspace{0.2cm}

\noindent\textbf{Acknowledgements.} The authors are deeply grateful to the referee for
 valuable comments and suggestions which have improved the paper considerably. In  particular, the referee conjectured
that the arguments and proof methods  should also work for Dunford-Schwartz operators, not just for Koopman operators as stated originally. 
We also thank
 the DAAD for support of the visit of the second author at the University of Leipzig. The second author has received funding from the European Research Council under the European Union's Seventh Framework Programme (FP7/2007-2013) / ERC grant agreement $\mr{n}^\circ$617747, and from the MTA R\'enyi Institute Lend\"ulet Limits of Structures Research Group.

\end{document}